\documentclass[11pt]{article}

\usepackage{amsmath,amsthm}

\usepackage{amssymb,latexsym}

\usepackage[mathscr]{euscript}

\usepackage{enumerate}

\newcommand{\BB}{{\cal B}}

\newcommand{\EE}{{\cal E}}
\newcommand{\FF}{{\cal F}}

\newcommand{\MM}{{\cal M}}

\newcommand{\RR}{{\cal R}}
\newcommand{\TT}{{\cal T}}

\newcommand{\BR}{{\mathbb R}}

\newcommand{\dyw}{\mbox{\rm div}}

\newtheorem{theorem}{\bf Theorem}[section]
\newtheorem{proposition}[theorem]{\bf Proposition}
\newtheorem{lemma}[theorem]{\bf Lemma}
\newtheorem{corollary}[theorem]{\bf Corollary}

\theoremstyle{definition}
\newtheorem{definition}[theorem]{Definition}
\newtheorem{example}[theorem]{\bf Example}

\newtheorem{remark}[theorem]{Remark}

\numberwithin{equation}{section}

\begin{document}

\title {On uniqueness and structure of renormalized solutions to
integro-differential equations with general measure data}
\author {Tomasz Klimsiak\\
{\small Faculty of Mathematics and Computer Science,
Nicolaus Copernicus University} \\
{\small  Chopina 12/18, 87--100 Toru\'n, Poland}\\
{\small E-mail address: tomas@mat.umk.pl}}
\date{}
\maketitle
\begin{abstract}
We propose  a new definition of renormalized solution to linear
equation with self-adjoint operator generating a Markov semigroup
and bounded Borel measure on the right-hand side. We give a
uniqueness result and study the structure of solutions to
truncated equations.
\end{abstract}




\section{Introduction}
\label{sec1}

In the paper, $E$ is a locally compact separable  metric space and
$m$ is a Radon measure on $E$ with full support. Let $(A,D(A))$ be
a non-positive definite self-adjoint operator on $L^2(E;m)$
associated with some Dirichlet form $(\EE,D(\EE))$ on $L^2(E;m)$.
The main goal of the present paper is to give a new definition of
renormalized solution to the linear equation
\begin{equation}
\label{eqi.1}
-Au=\mu
\end{equation}
with general (possibly nonsmooth in the Dirichlet forms  theory
sense) bounded Borel measure $\mu$  on $E$. It is known that such
a measure admits unique decomposition
\begin{equation}
\label{eq1.2} \mu=\mu_d+\mu_c
\end{equation}
into the  absolutely continuous, with respect to the  capacity
$\mbox{Cap}$ generated by $(\EE,D(\EE))$, part $\mu_d$ (so-called
diffuse part or smooth part of $\mu$) and the orthogonal, with
respect to Cap, part $\mu_c$ (so-called concentrated part). The problem of right definition to (\ref{eqi.1}) is rather subtle if we require that the solution $u$
be unique.

In the paper, we  assume that the resolvent
$(R_{\alpha})_{\alpha>0}$ generated by $A$ is Fellerian, i.e.
$R_\alpha (C_b(E))\subset C_b(E)$ for some (and hence for all)
$\alpha>0$, and   there exists  a Green function $G$ for  $A$ (see
Section \ref{sec2.4}).

Our new definition reads as follows: $u\in \BB(E)$ is a
renormalized solution to (\ref{eqi.1}) if
\begin{enumerate}
\item [(i)]  $T_k(u):= (u\wedge k)\vee(-k)\in D_e(\EE)$, $k\ge 0$,
where $D_e(\EE)$ is the extended Dirichlet space, i.e. an
extension of $D(\EE)$ such that $D_e(\EE)$ with inner product
$\EE$ is a Hilbert space.
\item [(ii)] There exists a family of bounded smooth  measures
$(\nu_k)_{k\ge 0}$ on $E$ such that
\[
\EE(T_k(u),\eta)=\langle \mu_d,\eta\rangle
+\langle \nu_k,\eta\rangle, \quad \eta\in D_e(\EE)\cap\BB_b(E),
\quad k\ge 0,
\]
\item[(iii)] $\nu_k\rightarrow \mu_c$ in the narrow topology, i.e.
for every $\eta\in C_b(E)$,
\[
\lim_{k\rightarrow\infty}\int_E\eta\,d\nu_k=\int_E\eta\,d\mu_c.
\]
\end{enumerate}

A similar definition of a solution to (\ref{eqi.1}), also
guaranteeing uniqueness, was  introduced recently in my joint
paper with Rozkosz \cite{KR:MM}. In that paper by a solution we
mean $u\in\BB(E)$ satisfying  (i) and (ii), and the following
condition
\begin{enumerate}
\item[(iii')] $\lim_{k\rightarrow\infty}\int_EG(x,y)\,\nu_k(dy)= \int_EG(x,y)\,\mu_c(dy)$ for $m$-a.e. $x\in E$.
\end{enumerate}
Condition (iii) is much simpler than (iii') because it does not involves the notion of the Green function. One of the main results of the present paper says that (i)--(iii) still ensure uniqueness for solutions to  (\ref{eqi.1}).
In the second part of the paper  we show interesting  properties of  the family $(\nu_k)_{k\ge 0}$: a structure theorem, so-called reconstruction property
and the narrow convergence of variations.

The above definition (i)--(iii) is a counterpart to the definition introduced by Dal Maso, Murat, Orsina and Prignet   \cite{DMOP} for equations with local nonlinear operators of Leray-Lions type of the form
\[
A(u)=\dyw(a(\cdot,\nabla u)).
\]
For such operators,  $\EE$ appearing in (i), (ii) is replaced by
\[
\EE(u,v):=\int_E a(\cdot,\nabla u)\nabla v\,dm,
\]
and the domain of $\EE$ is the natural energy space in which $\EE(u,u)$ is finite.
As a matter of fact, this modified definition (i)--(iii) is one of the four equivalent definitions of renormalized solutions considered in \cite{DMOP}.

The concept of renormalized solutions  was a crucial step in the development of the  theory of  elliptic and parabolic equations with (nonlinear) local operators and  measure data since  it gives  partial uniqueness results.  The complete uniqueness result in nonlinear case  is still an open problem.

It is worth noting here that among the defintions considered in \cite{DMOP} the defintion presented above have some remarkable feature. Namely, in condition (ii) the term $\EE(T_k(u),\eta)$ is well defined since both $T_k(u) $ and $\eta$ are in the  domain of $\EE$.  In the other  definitions considered in \cite{DMOP},
a different variational formulas (counterparts to (ii)) are considered. In these formulas the  term  $\EE(u,\eta)$ always appears which of course requires an extension
of the form $(\EE,D(\EE))$ in such a way that $\EE(u,\eta)$ makes sense for rich enough class of test functions $\eta$.  All the used extensions of $\EE$
in \cite{DMOP} are based   on the property
\[
\EE(u,T_k(u))=\EE(T_k(u),T_k(u)),
\]
which is true only for the forms associated with  local operators. For that reason only definition of type (i)--(iii) can be directly adopted to the
nonlocal case.  In case $\EE$ is nonlocal of the form
\begin{equation}
\label{eqi.7.1}
\EE(u,v)=\int_{E\times E}(u(x)-u(y))(v(x)-v(y))\,J(dx,dy)
\end{equation}
for some symmetric positive measure $J$ on $E\times E\setminus d$, Alibaud,  Andreianov and Bendahmane \cite{AAB} proposed the following extension of $\EE$: for $h\in C^1_c(\BR)$
\begin{align}
\label{eq.i1.2}
\nonumber
\EE(u,h(u)\eta)&:= \int_{E\times E}(u(x)-u(y))(h(u)(x)-h(u)(y))
\frac{\eta(x)+\eta(y)}{2}\,J(dx,dy)\\
& \quad+\int_{E\times E}(u(x)-u(y))(\eta(x)-\eta(y))
\frac{h(u)(x)+h(u)(y)}{2}\,J(dx,dy)
\end{align}
for bounded $\eta$  and $u$ such that
\begin{align}
\label{eq.i1.3}
\int_{E\times E}(u(x)-u(y))(T_k(u)(x)-T_k(u)(y))\,J(dx,dy)<\infty,
\quad k\ge 0.
\end{align}
A careful analysis shows that thanks to (\ref{eq.i1.3})  and
regularity of $h$ both integrals in (\ref{eq.i1.2}) are well
defined. However, the crucial assumption that $h$ has compact
support  makes this approach applicable only to equations with
smooth (diffuse) measure data. In \cite{AAB}, imitating one of the
definition considered in \cite{DMOP},  the authors introduced the
following definition of a solution to (\ref{eqi.1}) with $A$
generated  by (\ref{eqi.7.1}), $E=\BR^d$ and  smooth measure data
(in fact for $\mu\in L^1(\BR^d)$ but it naturally extends to
smooth measure data): a measurable function $u$ satisfying
(\ref{eq.i1.3}) is a renormalized solution to (\ref{eqi.1}) if
\begin{equation}
\label{eq.i1.5}
\EE(u,h(u)\eta)=\langle \mu,h(u)\eta\rangle,\quad \eta\in C^\infty(\BR^d),\, h\in C^1_c(\BR^d)
\end{equation}
and
\begin{align}
\label{eq.i1.4}
&\lim_{k\rightarrow\infty}\int_{E\times E} (u(x)-u(y))\{(T_{k+1}(u)-T_k(u))(x)
\nonumber\\
&\qquad\qquad\qquad-(T_{k+1}(u)-T_k(u))(y)\}\,J(dx,dy)=0.
\end{align}
We extend this definition to general forms $\EE$ considered here (and smooth measure data) and
show that if $u$ is a renormalized solution to (\ref{eqi.1}) in the sense  of definition (i)--(iii), then $u$ is a renormalized solution
to (\ref{eqi.1}) is the sense of  (\ref{eq.i1.5}) and (\ref{eq.i1.4}).

In order to get existence and uniqueness result for solutions to
(\ref{eqi.1}) with general measure data, we propose definition
(i)--(iii) which  seems to be the more suitable  formulation of
the definition  of renormalized solution to (\ref{eqi.1}) since it is
applicable not only to general measure data but also to  wide
class of operators associated  with  local and non-local  Dirichlet forms.

We prove that for bounded Borel measure $\mu$  there exists a
unique renormalized solution to (\ref{eqi.1}). Moreover, if $u$ is
renormalized solution to (\ref{eqi.1}) then
 \begin{equation}
 \label{eqi.2}
u(x)=\int_E G(x,y)\,\mu(dy)\quad m\mbox{-a.e. }x\in E,
 \end{equation}
and even stronger convergence of $\{\nu_k\}$ holds. Namely
\[
\nu_k^+\rightarrow \mu_c^+,\quad \nu_k^-\rightarrow \mu_c^-\quad
\mbox{narrowly.}
\]
From this it follows in particular that $u$ is a renormalized
solution to (\ref{eqi.1}) in the sense of  (i)--(iii) if and only
if $u$ is a duality solution to (\ref{eqi.1}) in the sense of
Stampacchia. The notion of duality solutions for linear equations
with uniformly elliptic divergence form operators and general
measure data was introduced by Stampacchia in \cite{Stampacchia}.
His approach was adapted to fractional Laplacian in
\cite{KPU,Petitta1}. The general formulation for operators $A$
generated by Markov semigroups was introduced in \cite{K:CVPDE}
(see also \cite{KR:JFA} for the case of smooth measure data).

In the second part of the paper, we give a complete
characterization  of the family $(\nu_k)_{k\ge 0}$.  Recall that
each regular symmetric Dirichlet form $(\EE,D(\EE))$ admits the
following (unique)  Beurling-Deny decomposition
\[
\EE(u,v)=\EE^{(c)}(u,v)+\int_{E\times
E}(u(x)-u(y))(v(x)-v(y))\,J(dx,dy) +\int_E v(x) u(x)\,\kappa(dx).
\]
Here $\EE^{(c)}$ is the local part of $\EE$, $J$ is a symmetric
positive Radon measure outside the diagonal $d$ of $E\times E$ and
$\kappa$ is a smooth Radon measure on $E$, called the killing
measure. We show that
\begin{equation}
\label{eqi.124}
\nu_k=-\mathbf{1}_{\{u\ge k, u<-k\}}\cdot \mu_d
+\frac12 l_k(u)-\frac12 l_{-k}(u)+
\frac12 j_k(u)-\frac12 j_{-k}(u)
\end{equation}
with
\begin{align*}
j_k(u)(dx)&=2\int_E
\Big(|u(y)-k|-|u(x)-k|-\mbox{\rm sign}(u(x)-k)(u(y)-u(x))\Big)J(dx,dy)\\
&\quad
+(\mathbf{1}_{\{u(x)>k\}}(|k|+k)+\mathbf{1}_{\{u(x)\le
k\}}(|k|-k))\kappa(dx).
\end{align*}
and $\{l_k(u),\, k\in\mathbb Z\}$ characterized as follows
\begin{equation}
\label{eq.ad.fd.oam}
\int_\BR\langle l_k(u),\eta\rangle \psi(k)\,dk=\langle \mu^c_{\langle u \rangle},\psi(u)\eta\rangle,\quad \psi,\eta\in\BB^+(E),
\end{equation}
where $\mu_{\langle u\rangle}^c$ is a  positive smooth  Radon measure
measure on $E$ given by
\begin{equation}
\label{eqi.oqwb4}
\langle \mu^c_{\langle u\rangle},\eta\rangle
=\lim_{k\rightarrow \infty}2\EE^{(c)}(T_k(u)\eta,T_k(u))
-\EE^{(c)}(T_k(u)^2,\eta),\quad \eta\in \BB_b(E)\cap D(\EE).
\end{equation}
From (\ref{eq.ad.fd.oam}) it follows in particular that if $\EE$
is local (i.e. $J\equiv 0$), then
\[
\frac{1}{c_n-b_n}\int_{\{b_n\le u\le c_n\}} \eta\, d\mu^c_{\langle u\rangle}\rightarrow \langle \mu_c^+,\eta\rangle,\qquad
\frac{1}{c_n-b_n}\int_{\{-c_n\le u\le -b_n\}}\eta \,d \mu^c_{\langle u\rangle}\rightarrow \langle \mu_c^-,\eta\rangle,
\]
for all  sequences $\{b_n\}$, $\{c_n\}$ of positive numbers such
that $b_n<c_n$, $n\ge1$ and $b_n,c_n\rightarrow \infty$ as
$n\rightarrow\infty$ (the so-called reconstruction property).
Observe also that $j_k(u)$ may be concentrated on the whole $E$.
Hence, contrary to the local case, the measures $\nu_k$ need not  be
concentrated on the set $\{|u|=k\}$.

In the present paper we focus our attention on linear equation
(\ref{eqi.1}). However, our results also apply to semilinear
equations of the form
\begin{equation}
\label{eqi.123}
-Au=f(\cdot,u)+\mu
\end{equation}
with $f$ being a measurable function on $E\times\BR$. By a
renormalized solution to (\ref{eqi.123}) we mean a measurable
function $u$ on $E$ such that $f(\cdot,u)\in L^1(E;m)$ and
(i)--(iii) hold when we replace  $\mu_d$ by $f(\cdot,u)+\mu_d$ in
condition (ii) (since $(f(\cdot,u)+\mu)_d=f(\cdot,u)+\mu_d$). From
our results it follows that if $f$ is nonincreasing with respect
to the second variable, then there exists at most one renormalized
solution to (\ref{eqi.123}).

\section{Notation and standing assumptions}
\label{sec2}

In the paper, $E$ is a locally compact separable metric space and $\partial$ is a one-point compactification of $E$. If $E$ is already compact, the $\partial$ is an isolated point. We adopt the convention that each function $f$ on $E$ is extended to $E\cup\{\partial\}$ by setting $f(\partial)=0$.

We denote by $\BB(E)$ the set of Borel measurable functions on
$E$.   $\BB_b(E)$, $\BB^+(E)$  are the subsets of $\BB(E)$
consisting of bounded and positive functions, respectively. We
denote by $\MM(E)$ the set of Borel measures on $E$, and by
$\MM^+(E)$ the subset of $\MM(E)$ consisting of positive measures.
For given $\eta\in \BB^+(E)$ and $\mu\in \MM^+(E)$ we set
\[
\langle \mu,\eta\rangle= \int_E\eta\,d\mu,
\]
whenever the integral exists. For given  $u \in \BB^+(E)$ and
$\mu\in \MM^+(E)$, we denote by $u\cdot \mu$ the measure on $E$
defined by
\[
\langle u\cdot\mu,\eta\rangle
=\langle \mu,u\eta\rangle,\quad \eta\in\BB^+(E).
\]

\subsection{Dirichlet forms and potential theory}

In the whole paper,  $(\EE,D(\EE))$ is a symmetric Dirichlet form
on $L^2(E;m)$ which  is regular, that is the set $C_c(E)\cap
D(\EE)$ is dense in $C_c(E)$ with uniform norm, and in $D(\EE)$
with $\EE_1$-norm, where for $\alpha>0$,
$\EE_\alpha(u,v)=\EE(u,v)+\alpha(u,v)$, $u,v\in D(\EE)$. We also
assume that it is transient, that is there exists a strictly
positive function $g$ on $E$ such that
\[
\int_E |u| g\,dm\le c\sqrt {\EE(u,u)},\quad u\in D(\EE).
\]
Therefore, by \cite[Theorem 1.5.3]{FOT}, there exists an extension
$D_e(\EE)\subset L^1(E,g\cdot m)$ of $D(\EE)$ such that
$(\EE,D_e(\EE))$ is a Hilbert space. By \cite[Theorem 1.3.1]{FOT},
there exists a unique self-adjoint non-positive definite operator
$(A,D(A))$ on $L^2(E;m)$ such that
\[
D(A)\subset D(\EE),\quad \EE(u,v)=(-Au,v),
\quad u\in D(A),\, v\in L^2(E;m).
\]
We denote by $(T_t)_{t\ge 0}$ the semigroup of contractions on $L^2(E;m)$
generated by $(A,D(A))$ and by $\mbox{Cap}$ the  capacity on $E$
defined  as follows: for and open $U\subset E$,
\[
\mbox{Cap}(U)=\inf\{\EE(u,u):u\in D(\EE),\, u\ge \mathbf{1}_U\,\,
m\mbox{-a.e.}\},
\]
and for an arbitrary $B\subset E$,
\[
\mbox{Cap}(B)=\inf\{\mbox{Cap}(U): B\subset U,\, U\subset
E\,\,\mbox{is open}\}.
\]
We say that a property $P$ holds quasi everywhere (q.e. in
abbreviation) if it holds outside a set $N$ with
$\mbox{Cap}(N)=0$. We say that a measurable function $u$ on $E$ is
quasi-continuous if for every $\varepsilon>0$ there exists a
closed set $F_\varepsilon\subset E$ such that
$\mbox{Cap}(E\setminus F_\varepsilon)<\varepsilon$ and
$u_{|F_\varepsilon}$ is continuous. By \cite[Theorem 2.1.7]{FOT},
every function $u\in D_e(\EE)$ has an $m$-version $\tilde u$ which
is quasi-continuous.

For a Borel measure on $E$,  $|\mu|$ stands for its total
variation. We say that a Borel measure $\mu$ on $E$ is smooth if
$|\mu|(B)=0$ for every Borel set $B\subset E$ such that
$\mbox{Cap}(B)=0$, and there exists a strictly positive
quasi-continuous function $\eta$ on $E$ such that $\langle
|\mu|,\eta\rangle<\infty$. The set of all positive smooth measures
on $E$ will be denoted by $S$. We denote by $\MM_{0,b}(E)$ the set
of bounded smooth measures, and by $\MM_{0,b}^+(E)$ the set of
positive bounded smooth measures. Each $\mu\in\MM(E)$ admits
unique decomposition of the form (\ref{eq1.2}), where $\mu_d$ is a
smooth measure and $\mu_c$ is concentrated on the set $B\subset E$
such that $\mbox{Cap}(B)=0$.

\subsection{Probabilistic potential theory}
\label{sec2.4}

By \cite[Section 7]{FOT}, there exists a Hunt process
$
\mathbb X=\{(X_t)_{t\ge 0}, (P_x)_{x\in E\cup {\partial}}, \mathbb F:= (\FF_t)_{t\ge 0}\}$
with life time $\zeta$
associated with the Dirichlet form $(\EE,D(\EE))$ in the sense that for every  $\eta\in L^2(E;m)\cap \BB^+(E)$,
\[
T_t f(x)=E_x f(X_t),\quad   m\mbox{-a.e.}
\]
In the paper, we assume that $\mathbb X$ satisfies absolute continuity condition, i.e. there exists a positive Borel function $p$ on $\BR^+\times E\times E$
such that for all $x\in E$, $t>0$ and $f\in\BB^+(E)$,
\[
E_x f(X_t)=\int_E f(y) p(t,x,y)\, m(dy).
\]
We denote by $(P_t)_{t\ge 0}$ (resp. $(R_{\alpha})_{\alpha>0})$ the semigroup (resp. resolvent) associated with $\mathbb X$. Recall that
for all $t,\alpha\ge 0$ and $f\in\BB^+(E)$,
\[
P_t f(x)=E_xf(X_t),\qquad R_\alpha f(x)= E_x\int_0^\zeta e^{-\alpha t}f(X_t)\,dt, \quad x\in E,
\]
We set $R:= R_0$. For  $\mu\in\MM^+(E)$, we set
\[
R\mu(x)=\int_E G(x,y)\,\mu(dy),\quad x\in E.
\]
Here $G$ is  the Green  function defined by
\[
G(x,y)=\int_0^\infty p(t,x,y)\,dt,\quad x,y\in E.
\]

In the sequel, we write that some property $P$ holds q.a.s. (resp. a.s.) if it holds outside  a set  $B\in\FF_\infty$
such that $P_x(B)=0$ for q.e. $x\in E$ (resp. for every $x\in E$).

Recall that $u\in\BB^+(E)$ is called excessive if
\[
\sup_{t>0} P_tu(x)=u(x),\quad x\in E.
\]
The set of all  excessive functions will be denoted by $Exc$. It is well known that for  any $\eta,\psi\in Exc$ the mapping
\[
\mathbb R^+\ni t\mapsto \frac1t\langle \psi,\eta-P_t\eta\rangle
\]
is nonincreasing. Moreover, if $\psi=R\mu$ for some positive Borel measure $\mu$ and $\psi<\infty$ $m$-a.e., then
\[
\sup_{t>0}\frac1t\langle \psi,\eta-P_t\eta\rangle=\lim_{t\rightarrow 0^+}\frac1t\langle \psi,\eta-P_t\eta\rangle=\langle\mu,\eta\rangle.
\]

Let $\mathbf A^+_c$ denote the set of  positive continuous additive functionals of $\mathbb X$ (see \cite[Section 5.1]{FOT}). It is well known that there exists an isomorphism (the so-called Revuz duality)
\[
\RR: \mathbf A^+_c-\mathbf A_c^+\rightarrow S-S
\]
defined as follows: for every $A\in \mathbf A^+_c$ and every continuous positive $\eta$ on $E$,
\[
\langle \mathcal R(A),  \eta\rangle = \lim_{t\rightarrow 0^+}\frac1t E_m\int_0^t \eta(X_r)\,dA_r.
\]
By \cite[Theorem 5.1.3]{FOT}, $t\mapsto \frac1t E_m\int_0^t \eta(X_r)\,dA_r$ is   nonincreasing  for any  $\eta\in Exc$ and $A\in\mathbf A^+_c$, and the above equality also holds for $\eta\in Exc$. Moreover,  if $\eta$ is excessive or is a positive continuous function $\eta$ on $E$, then
\begin{equation}
\label{eq2.1}
\lim_{t\rightarrow 0^+}\frac1t E_m\int_0^t \eta(X_r)\,dA_r=\lim_{t\rightarrow 0^+}\frac1t E_{\eta\cdot m}\int_0^t \,dA_r.
\end{equation}
In the sequel, for given $\mu\in S$ we set $A^\mu:= \mathcal R^{-1}(\mu)$.
By \cite[Theorem 5.1.3]{FOT}, for all $\eta\in \BB^+(E)$ and $\mu\in S$,
\begin{equation}
\label{eq2.rdgf}
E_x\int_0^\zeta \eta(X_t)\,dA^\mu_t=\int_E\eta(y)G(x,y)\,\mu(dy),\quad x\in E.
\end{equation}

For a given c\`adl\`ag special semimartingale $Y$ and $k\in \BR$,
we denote by $L^k(Y)$ the local time of $Y$ at $k$ (see \cite[page
212]{Protter}). We also put
\[
J^k_t(Y)= \sum_{s\le t}
(|Y_s-k|-|Y_{s-}-k|-\mbox{sign}(Y_{s-}-k)\Delta Y_s ),
\]
where $\Delta Y_t=Y_t-Y_{t-}$\,, $Y_{t-}=\lim_{s\rightarrow t^-}
Y_s$ and $\mbox{sign}(x)=1$ if $x>0$ and $\mbox{sign}(x)=-1$ if
$x\le0$. Since $x\mapsto|x-k|$ is a convex function, $J^k$ is an
increasing process. By the Tanaka-Meyer formula (see \cite[Theorem
70, page 214]{Protter}) for every convex function
$\varphi:\BR\rightarrow \BR$,
\begin{align}
\label{eq2.tmf} \varphi(Y_t)=\varphi(Y_0)+\int_0^t
\varphi'(Y_{s-})\,dY_s+\frac12\int_\BR L^k_t(Y)\,\mu(dk)+\frac12\int_\BR
J^k_t(Y)\,\mu(dk),
 \end{align}
where $\varphi'$ is the left derivative of $\varphi$ and $\mu= (\varphi')'$ with second derivative taken in distributional sense.
Since $Y$ is a special semimartingale, there exists process
${}^pJ^k(Y)$ which is the dual predictable projection of $J^k(Y)$.
In case $Y=u(X)$ q.a.s. it is easy to observe that $L^k(u(X))$ and $J^k(u(X))$ are positive additive functionals.
By the definition of local times, $L^k(u(X))$ is continuous, and since $\mathbb X$ is a Hunt process, $^pJ^k(u(X))$ is continuous too (and it is still a positive additive functional, see \cite[Theorem A.3.16]{FOT}). We set
\begin{equation}
\label{eq2.rdgf1}
l_k(u)=\RR(L^k(u(X))),\quad    j_k(u)=\RR(^pJ^k(u(X))),\quad \lambda_k(u)=l_k(u)+j_k(u).
\end{equation}
By \cite[Theorem IV.69]{Protter}, the measure $l_k(u)$ is concentrated on the set $\{u=k\}$.

\section{Integral, duality,  probabilistic  and very weak solutions}
\label{sec3}

\begin{definition}
\label{def2.1i}
Let $\mu\in\MM_{b}(E)$.
We say that $u\in \mathcal B(E)$ is an integral solution to (\ref{eqi.1}) if
\begin{equation}
\label{eq2.irf}
u(x)=\int_E G(x,y)\,\mu(dy),\quad m\mbox{-a.e. }x\in E.
\end{equation}
\end{definition}

\begin{remark}
\label{rem.aqc}
Let $u$ be an integral solution to (\ref{eqi.1}). Define
\[
N=\{x\in E:\int_E G(x,y)\,|\mu|(dy)=\infty\}.
\]
By \cite[Proposition 3.2]{K:CVPDE}, $\mbox{Cap}(N)=0$, and   by \cite[Theorem 3.7]{K:CVPDE}, $\tilde u$ defined as
\begin{equation}
\label{eq3.etfr}
\tilde u(x)=0,\quad x\in N,\qquad \tilde u(x)=\int_E G(x,y)\,\mu(dy),\quad x\in E\setminus N,
\end{equation}
is a quasi-continuous $m$-version of $u$. Therefore  we may assume that any integral solution $u$ to (\ref{eqi.1}) is quasi-continuous and (\ref{eq2.irf}) is satisfied for  q.e. $x\in E$.
\end{remark}

The next definition introduced in \cite{K:CVPDE}  is a generalization, to the class of operators considered in the present paper, of Stampacchia's definition  introduced in \cite{Stampacchia} in case $A$ is a uniformly elliptic diffusion operator in divergence form. In case $A$ is a fractional Laplacian, duality solutions  were considered in \cite{KPU} (on $\BR^d$) and in \cite{Petitta1} (on bounded domains in $\BR^d$).

\begin{definition}
\label{def2.1d}
Let $\mu\in\MM_{b}(E)$.
We say that $u\in \mathcal B(E)$ is a duality solution to (\ref{eqi.1}) if
for every $\eta\in\BB(E)$ such that $R|\eta|$ is bounded
\[
\langle u, \eta\rangle=\langle \mu, R\eta\rangle.
\]
\end{definition}

It is worth noting here that in case $\mu$ is a smooth measure it is possible to define duality solutions for general operators corresponding to transient regular Dirichlet forms, i.e. without the additional assumption that there exists the Green function for $A$ (see \cite{KR:JFA}).

The following definition of probabilistic solution to (\ref{eqi.1}) was introduced in \cite{K:CVPDE}. To formulate it, we first recall that $M$ is called a local
martingale additive functional (local MAF for short) of $\mathbb X$ if $M$ is an additive functional of $\mathbb X$ and $M$ is a local martingale under the measure $P_x$ for q.e. $x\in E$ (see  \cite{K:CVPDE} for details).

\begin{definition}
\label{def2.1}
Let $\mu\in\MM_{b}(E)$.
We say that $u\in \mathcal B(E)$ is a probabilistic solution to (\ref{eqi.1}) if
\begin{enumerate}
\item [(i)] there exists a local martingale additive functional $M$ of $\mathbb X$ such that
\[
u(x)=u(X_t)+\int_0^t\,dA^{\mu_d}_r-\int_0^t\,dM_r,\quad t\in [0,\zeta),\quad\mbox{q.a.s.}
\]
\item [(ii)] for every sequence  $\{\tau_k\}$ of $\mathbb F$-stopping times such that $E_x\sup_{t\le\tau_k}|u(X_t)|<\infty$ for q.e. $x\in E$  and $\tau_k\nearrow \zeta$ q.a.s. we have
\[
E_x(u(X_{\tau_k}))\rightarrow R\mu_c(x)\quad\mbox{q.e. }x\in E.
\]
\end{enumerate}
\end{definition}

Any sequence $\{\tau_k\}$ of $\mathbb F$-stopping times such that $E_x\sup_{t\le\tau_k}|u(X_t)|<\infty$ for q.e. $x\in E$ and $\tau_k\nearrow \zeta$ q.a.s.
is called the reducing sequence for $u$.

\begin{remark}
From  \cite[Proposition 3.12]{K:CVPDE} it follows that any probabilistic solution to (\ref{eqi.1}) is quasi-continuous.
\end{remark}

By \cite[Propositions 3.12, 4.12]{K:CVPDE} we have the following result.

\begin{proposition}
\label{prop3.eq1}
Let $\mu\in\MM_b(E)$.
\begin{enumerate}
\item[\rm(i)] If $u$ is a probabilistic solution to \mbox{\rm(\ref{eqi.1})}, then $u$ is an integral solution to \mbox{\rm(\ref{eqi.1})}.
\item[\rm(ii)] If $u$ is an integral solution to \mbox{\rm(\ref{eqi.1})}, then $\tilde u$ defined by \mbox{\rm(\ref{eq3.etfr})} is a probabilistic solution to \mbox{\rm(\ref{eqi.1})}.
\item[\rm(iii)] $u$ is an integral solution to \mbox{\rm(\ref{eqi.1})} iff it is  a duality solution to \mbox{\rm(\ref{eqi.1})}.
\end{enumerate}
\end{proposition}

From now on, we always consider quasi-continuous versions of  solutions to (\ref{eqi.1}) (no matter which one of the definition we consider).

\begin{proposition}
\label{prop3.1}
Let $\mu\in\MM_b(E)$ and $u$ be an integral solution to  (\ref{eqi.1}). Then $\lambda_k(u)\in \MM_{0,b}(E)$  for every $k\in\BR$. Moreover,
\begin{equation}
\label{eq3.7.1}
-Au^+= \mathbf{1}_{\{u>0\}}\cdot \mu_d-\frac12 \lambda_0(u)+\mu_c^+,
\end{equation}
and  for  every $k>0$,
\begin{equation}
\label{eq3.7.2}
-A(u^+\wedge k)=\mathbf{1}_{\{0<u\le k\}}\cdot\mu_d+\frac12\lambda_k(u)-\frac12\lambda_0(u),
\end{equation}
\begin{equation}
\label{eq3.7.3}
-A(u^-\wedge k)=-\mathbf{1}_{\{-k<u\le 0\}}\cdot\mu_d+\frac12\lambda_{-k}(u)-\frac12\lambda_0(u).
\end{equation}
\end{proposition}
\begin{proof}
By Proposition \ref{prop3.eq1} and the definition of a probabilistic solution to (\ref{eqi.1}),
\begin{equation}
\label{eq3.1}
u(X_t)=u(X_{\tau_n})+\int_t^{\tau_n} \,dA^{\mu_d}_r-\int_t^{\tau_n}\,dM_r,\quad t\in [0,\tau_n],\, \rm{q.a.s.,}
\end{equation}
for some local MAF $M$ of $\mathbb X$ and reducing sequence $\{\tau_n\}$ for $u(X)$.
By the Tanaka-Meyer formula,
\begin{align*}
u^+(X_t)&=u^+(X_{\tau_n})+\int_t^{\tau_n}\mathbf{1}_{\{u(X_{r-})>0\}} \,dA^{\mu_d}_r
-\frac12\int_t^{\tau_n}\,dL^0_r(u(X))\\
&\quad-\frac12\int_t^{\tau_n}\,dJ^0_r(u(X))-\int_t^{\tau_n}\mathbf{1}_{\{u(X_{r-})>0\}} \,dM_r,\quad t\in [0,\tau_n],\,\rm{q.a.s.}
\end{align*}
Hence
\begin{align*}
u^+(X_t)&=u^+(X_{\tau_n})+\int_t^{\tau_n}\mathbf{1}_{\{u(X_{r-})>0\}} \,dA^{\mu_d}_r
-\frac12\int_t^{\tau_n}\,dL^0_r(u(X))\\
&\quad-\frac12\int_t^{\tau_n}\,d( ^p[J^0_r(u(X))])-\int_t^{\tau_n} \,dN_r,\quad t\in [0,\tau_n],\,\rm{q.a.s.}
\end{align*}
for some local MAF $N$ of $\mathbb X$.
Taking the expectation of  both sides of the above equation and then letting $n\rightarrow \infty$ and using (\ref{eq2.rdgf}) and \cite[Theorem 6.3]{K:CVPDE} we  obtain
\[
u^+(x)=R(\mathbf{1}_{\{u>0\}}\cdot\mu_d)(x)-\frac12R(\lambda_0(u))(x)+R\mu^+_c(x),\quad \mbox{q.e. }x\in E.
\]
This implies (\ref{eq3.7.1}). Furthermore, (\ref{eq3.7.1}) when combined with
\cite[Lemma 4.6]{K:CVPDE} implies that $\lambda_0(u)\in \MM_b(E)$. Applying the Tanaka-Meyer formula to (\ref{eq3.1}) (with the function $\varphi(x)=x^+\wedge k$) yields
\begin{align*}
(u^+\wedge k)(X_t)&=(u^+\wedge k)(X_{\tau_n})+\int_t^{\tau_n}\mathbf{1}_{\{0<u(X_{r-})\le k\}} \,dA^{\mu_d}_r-\frac12\int_t^{\tau_n}\,dL^0_r(u(X))\\
&\quad-\frac12\int_t^{\tau_n}\,d( ^p[J^0_r(u(X))])
+\frac12\int_t^{\tau_n}\,dL^k_r(u(X))\\
&\quad+\frac12\int_t^{\tau_n}\,d( ^p[J^k_r(u(X))])-\int_t^{\tau_n} \,dN_r,\quad t\in [0,\tau_n],\,\rm{q.a.s.}
\end{align*}
for some MAF $N$ of $\mathbb X$. Since $u^+\wedge k$ is bounded,  $E_x(u^+\wedge k)(X_{\tau_n})\rightarrow 0$ for q.e. $x\in E$. Therefore taking the expectation of  both sides of the above equation and applying (\ref{eq2.rdgf}) yields
\begin{equation}
\label{eq3.2}
(u^+\wedge k)(x)=R(\mathbf{1}_{\{0<u\le k\}}\cdot\mu_d)(x)-\frac12R(\lambda_0(u))(x)+\frac12R(\lambda_k(u))(x)
\end{equation}
for q.e. $x\in E$, which shows (\ref{eq3.7.2}). Since $u^+\wedge k\le u^+$, we have
\begin{align*}
& R(\mathbf{1}_{\{0<u\le k\}}\cdot\mu_d)-\frac12R(\lambda_0(u))+\frac12R(\lambda_k(u))\\
&\qquad \le R(\mathbf{1}_{\{u>0\}}\cdot\mu_d)-\frac12R(\lambda_0(u))+R\mu^+_c\quad \mbox{q.e.}
\end{align*}
By this  and \cite[Lemma 4.6]{K:CVPDE} again,  $\lambda_k(u)\in \MM_b(E)$ for $k>0$. Using the same argument, but with the function $\varphi(x)=x^-\wedge k$, we show that
$\lambda_k(u)\in \MM_b(E)$ for $k<0$ and  (\ref{eq3.7.3}) is satisfied.
\end{proof}

\begin{corollary}
\label{eq3.sr1}
Let $u$ be an integral solution to \mbox{\rm(\ref{eqi.1})}. Then
\[
\nu_k=-\mathbf{1}_{\{u>k, u\le -k\}}\mu_d+\frac12(\lambda_k(u)-\lambda_{-k}(u)),\quad k>0.
\]
\end{corollary}
\begin{proof}
Follows from Proposition \ref{prop3.1} and the fact that $T_k(u)=u^+\wedge k-u^-\wedge k$.
\end{proof}

Now we are going to show  that any integral solution to (\ref{eqi.1})  is the so-called very weak solution to (\ref{eqi.1}). Set
\[
F:=\{\eta\in D(A):-A\eta\in L^{\infty,+}(E;m)\}.
\]
In the definition of very weak solution we require that  test functions are defined in each point of $E$ and not only $m$-a.e. or q.e.
In many cases one would take $F\cap C_b(E)$. However, in general, it may happen that $D(A)\cap C_b(E)=\{0\}$. Therefore as test functions we take excessive $m$-versions of elements of $F$. Such versions are finely-continuous, so are defined everywhere. That each $\eta\in F$ have an excessive $m$-version $\check\eta$ follows form \cite[Lemma 2.1.1]{FOT}.

\begin{definition}
\label{def.4.vws}
We say that $u\in L^1(E;m)$ is a very weak solution to (\ref{eqi.1}) if
\[
\langle u,-A\eta\rangle=\langle \mu,\check \eta\rangle,\quad \eta\in F-F.
\]
\end{definition}

\begin{lemma}
\label{lm3.5}
Let $\eta\in D(A)$ be a bounded excessive function. Then
\[
R(-A\eta)=\eta.
\]
\end{lemma}
\begin{proof}
Since $\eta\in Exc$, $-A\eta\ge 0$.
By the very definition of the operator $R$,
\[
R(-A\eta)=\sup_{\alpha>0}R_\alpha(-A\eta)=\lim_{\alpha\searrow 0} R_\alpha(-A\eta).
\]
By the resolvent identity,
\[
R_\alpha(-A\eta)=\lim_{\beta\rightarrow \infty} \beta R_\alpha(\eta-\beta R_\beta\eta)= \lim_{\beta\rightarrow \infty}( \beta R_\beta\eta -\alpha\beta R_\beta R_\alpha \eta).
\]
Since $\eta$ is bounded and excessive,
\[
\lim_{\beta\rightarrow \infty}( \beta R_\beta\eta -\alpha\beta R_\beta R_\alpha \eta)=\eta-\alpha R_\alpha\eta,
\]
so $R_\alpha(-A\eta)=\eta-\alpha R_\alpha\eta$.
Since $\eta$ is bounded and excessive and $\eta\in D(A)$, there exists $\xi\in\BB^+(E)$
such that $\eta=R\xi$. Hence
$\alpha R_\alpha\eta=\alpha R_\alpha R\xi=R\xi-R_\alpha\xi\rightarrow0$
as $\alpha\searrow0$. Thus $\lim_{\alpha\searrow 0} R_\alpha(-A\eta)=\eta$, and the proof is complete.
\end{proof}

\begin{proposition}
\label{prop3.4}
Let $\mu\in \MM_{b}(E)$ and $u\in L^1(E;m)$ be an integral   solution to \mbox{\rm(\ref{eqi.1})}. Then $u$ is a very weak solution to \mbox{\rm(\ref{eqi.1})}.
\end{proposition}
\begin{proof}
Let $\eta\in F$. By the remark preceding Definition \ref{def.4.vws}, there exists an $m$-version  $\check \eta$ of $\eta$ such that $\check \eta$ is excessive.
Applying Lemma \ref{lm3.5} we have
\[
\langle u,-A\eta\rangle =\langle R\mu,-A\check \eta\rangle =\langle \mu,R(-A\check \eta)\rangle =\langle \mu,\check\eta\rangle,
\]
which proves the proposition.
\end{proof}

\section{Renormalized solutions for general measure data}
\label{sec4}

In this section, we consider two equivalent definitions of renormalized solution to (\ref{eqi.1}) and we study their relations to other concepts of solutions considered in Section \ref{sec3}.  The first definition was introduced in \cite{KR:MM}. The second is new. Its advantage over the first one is that it is simpler because it does not involve using the notion of the potential of the measure.

\subsection{First definition}

\begin{definition}
\label{def2.2}
Let $\mu\in\MM_b(E)$. A Borel measurable function $u$ on $E$ is called a renormalized solution to (\ref{eqi.1}) if
\begin{enumerate}
\item [(i)] $T_k(u)\in D_e(\EE)$ for every $k>0$,
\item [(ii)] for every $k>0$ there exists   $\{\nu_k\}\subset \MM_{0,b}(E)$ such that
\[
\EE(T_k(u),\eta)=\langle \mu_d,\eta\rangle+\langle \nu_k,\eta\rangle,\quad \eta\in D_e(\EE)\cap \BB_b(E),
\]
\item[(iii)]  $R\nu_k\rightarrow R\mu_c$ $m$-a.e.
\end{enumerate}
\end{definition}

\begin{remark}
\label{rem4.tvc}
In \cite{KR:NoDEA} it is shown that if $\mu\in\MM_{0,b}(E)$ and $u$ is a renormalized solution to (\ref{eqi.1}), then $\|\nu_k\|_{TV}\rightarrow 0$ as $k\rightarrow\infty$, where $\|\nu_k\|_{TV}=|\nu_k|(E)$.
\end{remark}

By \cite[Theorem 4.4]{KR:MM}, each renormalized solution to (\ref{eqi.1}) has an $m$-version which is quasi-continuous. From now on we always consider
quasi-continuous versions of renormalized solutions to (\ref{eqi.1}).

\begin{proposition}
\label{prop4.eq1}
Let $\mu\in\MM_b(E)$. Then
\begin{enumerate}
\item[\rm(i)]  $u$ is a renormalized  solution to \mbox{\rm(\ref{eqi.1})} if and only if   it is an integral solution to \mbox{\rm (\ref{eqi.1})}.
\item[\rm(ii)]  There exists at most one renormalized solution to \mbox{\rm(\ref{eqi.1})}.
\end{enumerate}
\end{proposition}
\begin{proof}
Follows from Proposition \ref{prop3.eq1} and  \cite[Theorem 4.4]{KR:MM}.
\end{proof}

Our goal is to show that condition (iii) in Definition \ref{def2.2} may be replaced by the following condition (iii'): $\nu_k\rightarrow \mu_c$ in the narrow topology.

\begin{proposition}
\label{prop3.2}
Let $\mu\in\MM_b(E)$ and $u$ be a renormalized solution to \rm{(\ref{eqi.1})}. Then for every bounded excessive function $\eta$,
\[
\lim_{k\rightarrow\infty}\langle \frac12\lambda_k(u),\eta\rangle=\langle \mu^+_c,\eta\rangle,\qquad
\lim_{k\rightarrow\infty}\langle\frac12\lambda_{-k}(u),\eta\rangle=\langle \mu^-_c,\eta\rangle.
\]
\end{proposition}
\begin{proof}
We will prove the first assertion. The proof of the  second one is analogous. Write $\nu^1_k:= -\mathbf{1}_{\{u>k\}}\cdot\mu_d+\frac12\lambda_k(u)-\frac12\lambda_0(u)$. By Proposition \ref{prop3.1},
\[
-A(u^+\wedge k)=\mathbf{1}_{\{u>0\}}\cdot\mu_d+\nu^1_k.
\]
Let $\eta$ be a bounded excessive function. From the above equation,  Proposition \ref{prop4.eq1} and Revuz duality we conclude that
\begin{align*}
\langle\nu^1_k,\eta\rangle&=\lim_{t\rightarrow 0^+}\frac1t E_{\eta\cdot m}\int_0^t \,dA^{\nu^1_k}_r=\lim_{t\rightarrow 0^+}
\frac1t \langle R\nu^1_k-P_tR\nu^1_k,\eta\rangle\\&
=\lim_{t\rightarrow 0^+}\Big[-\frac1t\langle R(\mathbf{1}_{\{u>0\}}\cdot\mu_d)-P_tR(\mathbf{1}_{\{u>0\}}\cdot\mu_d),\eta\rangle\\
&\qquad\qquad\qquad+\frac1t \langle u^+\wedge k-P_t(u^+\wedge k),\eta\rangle\Big]
\\&
=\lim_{t\rightarrow 0^+}\Big[-\frac1tE_{\eta\cdot m}\int_0^t\,dA^{\mathbf{1}_{\{u>0\}}\cdot \mu_d}_r+\frac1t \langle u^+\wedge k,\eta-P_t\eta\rangle\Big]\\&
=-\langle\mathbf{1}_{\{u>0\}}\cdot \mu_d,\eta\rangle+\lim_{t\rightarrow 0^+}\frac1t \langle u^+\wedge k,\eta-P_t\eta\rangle .
\end{align*}
Since $\eta\in Exc$, we deduce from the above equation that $\langle\nu^1_k,\eta\rangle$ is nondecreasing. Therefore we may pass to the limit in the above equation as $k\rightarrow \infty$. We then have
\begin{align}
\label{eq4.1}
&\lim_{k\rightarrow \infty}\lim_{t\rightarrow 0^+}\frac1t \langle u^+\wedge k,\eta-P_t\eta\rangle\nonumber \\ &\quad
=\lim_{k\rightarrow \infty}\lim_{t\rightarrow 0^+}\frac1t\langle u^+\wedge k+R(\mathbf{1}_{\{0<u\le k\}}\cdot\mu^-_d)+\frac12 R(\lambda_0(u)),\eta-P_t\eta\rangle
\nonumber \\
&\qquad-\lim_{k\rightarrow \infty}\lim_{t\rightarrow 0^+}\frac1t\langle R(\mathbf{1}_{\{0<u\le k\}}\cdot\mu^-_d)+\frac12 R(\lambda_0(u)),\eta-P_t\eta\rangle.
\end{align}
It is clear that $R(\mathbf{1}_{\{0<u\le k\}}\cdot\mu^-_d)+\frac12 R(\lambda_0(u))$ is an excessive function.
By (\ref{eq3.2}), $u^+\wedge k+R(\mathbf{1}_{\{0<u\le k\}}\cdot\mu^-_d)+\frac12 R(\lambda_0(u))$ is also an excessive function. Therefore both limits with respect to $t$ on the right-hand side of (\ref{eq4.1}) are nondecreasing.
It is clear that this is also true for limits with respect to $k$. Therefore  we may change the order of the limits in (\ref{eq4.1}). By Proposition \ref{prop3.1}, we then have
\begin{align*}
\lim_{k\rightarrow \infty}\lim_{t\rightarrow 0^+}\frac1t \langle u^+\wedge k,\eta-P_t\eta\rangle &=
\lim_{t\rightarrow 0^+}\lim_{k\rightarrow \infty}\frac1t \langle u^+\wedge k,\eta-P_t\eta\rangle\\
&=\lim_{t\rightarrow 0^+}\frac1t \langle u^+,\eta-P_t\eta\rangle\\
&=\lim_{t\rightarrow 0^+}\frac1t\langle u^+-P_tu^+,\eta\rangle\\
&=\langle \mathbf{1}_{\{u>0\}}\cdot \mu_d-\frac12 \lambda_0(u)+\mu_c^+,\eta\rangle.
\end{align*}
This proves the first assertion.
\end{proof}


\begin{lemma}
\label{lm2.1}
Let $h\in C_b(E)$. Then there exist sequences $\{h^1_n\}, \{h^2_n\}\in \mbox{Exc}-\mbox{Exc}$ such that
\begin{enumerate}
\item[\rm{(i)}] $-\|h\|_\infty\le h^2_n\le h\le h^1_n\le \|h\|_\infty$, $n\ge1$,
\item[\rm{(ii)}] $h^2_n\nearrow h$, $h^1_n\searrow h$ as $n\rightarrow\infty$.
\end{enumerate}
\end{lemma}
\begin{proof}
Since $(\EE, D(\EE))$ is transient, there exists a strictly positive bounded function $g_0$ such that
$Rg_0$ is bounded (see \cite[Corollary 1.3.6.]{Oshima}). Set $g:= R_1 g_0$. Then $Rg=RR_1g_0=R_1Rg_0\le Rg_0=g$. It is clear that $g$ is bounded,  finely-continuous and strictly positive. For $n\ge1$, we set
\[
h^1_n(x)=\sup_{\tau\in\TT}E_x\Big(-n\int_0^{\tau}g(X_r)\,dr+h(X_\tau)\Big),\quad h^2_n(x)=\inf_{\tau\in\TT}E_x\Big(n\int_0^{\tau}g(X_r)\,dr+h(X_\tau)\Big),
\]
where $\TT$ is the set of all $\mathbb F$-stopping times.  We have
\[
h^2_n(x)=-\sup_{\tau\in\TT}E_x(-n\int_0^{\tau}g(X_r)\,dr+(-h(X_\tau))),
\]
so $-h^2_n$ is defined as  $h^1_n$ but with $h$ replaced by $-h$. Therefore it is enough to prove that $\{h^1_n\}$ has the desired properties. Observe that
\[
h^1_n(x)+nRg(x)=\sup_{\tau\in\TT}E_x((nRg+h)(X_\tau))
\]
Hence, by \cite{Nagai}, $h^1_n+nRg$ is an excessive function. It is clear that $nRg$ is also excessive. Thus $h^1_n\in \mbox{Exc}-\mbox{Exc}$.  By the definition, $\{h^1_n\}$ is nonincreasing. Moreover, with $\tau_0=0$, we have
\begin{align*}
h(x)&=E_x\Big(-n\int_0^{\tau_0}g(X_r)\,dr+h(X_{\tau_0})\Big)\\&
\le h^1_n(x)= \sup_{\tau\in\TT}E_x\Big(-n\int_0^{\tau}g(X_r)\,dr+h(X_\tau)\Big)\le \|h\|_\infty.
\end{align*}
Since $Rg$ is an excessive function and $h$ is continuous, the   process $(nRg+h)(X)$ is c\`adl\`ag
under the  measure $P_x$ for every $x\in E$. For every $\varepsilon>0$ there exists $\tau^x_{n,\varepsilon}\in\TT$ such that
\begin{align}
\label{eq3.3}
&E_x(-n\int_0^{\tau^x_{n,\varepsilon}}g(X_r)\,dr+h(X_{\tau^x_{n,\varepsilon}}))
-\varepsilon\nonumber\\
&\qquad \le h^1_n(x)\le E_x\Big(-n\int_0^{\tau^x_{n,\varepsilon}}g(X_r)\,dr
+h(X_{\tau^x_{n,\varepsilon}})\Big)+\varepsilon.
\end{align}
From this we conclude  that
\begin{equation}
\label{eq4.cont1}
nE_x\int_0^{\tau^x_{n,\varepsilon}}g(X_r)\,dr\le 2\|h\|_\infty+\varepsilon.
\end{equation}
Assume for a moment that we know that the above inequality implies that  $\tau^x_{n,\varepsilon}\rightarrow 0$ in probability $P_x$ as $n\rightarrow \infty$. Then, by continuity of $h$ and (\ref{eq3.3}),
\[
 h^1_n(x)\le E_x(-n\int_0^{\tau^x_{n,\varepsilon}}g(X_r)\,dr+h(X_{\tau^x_{n,\varepsilon}}))+\varepsilon\le E_x(h(X_{\tau^x_{n,\varepsilon}}))+\varepsilon\rightarrow h(x)+\varepsilon.
\]
This implies that $\lim_{n\rightarrow \infty} h^1_n(x)\le h(x)$. Since, $h^1_n\ge h$, we get the desired result. What is left is to show that $\tau^x_{n,\varepsilon}\rightarrow 0$ in probability $P_x$ as $n\rightarrow \infty$. Aiming for a contradiction, suppose that there exist $\varepsilon_1,\varepsilon_2>0$ and a subsequence (still denoted by $(n)$) such that
\[
P_x(\tau^x_{n,\varepsilon}>\varepsilon_1)>\varepsilon_2,\quad n\ge 1.
\]
Since $g$ is strictly positive, there exists $\delta>0$ such that $g(x)\ge2 \delta$.
Set
\[
\sigma^x=\inf\{t>0: |g(X_t)-g(x)|>\delta\}.
\]
Since $g$ is finely-continuous, for any sequence $t_n\searrow 0$,
$\lim_{n\rightarrow\infty}P_x(\sigma^x>t_n)=1$.
Let $t_{n_0}>0$ be such that $P_x(\sigma^x>t_{n_0})\ge 1-\frac{\varepsilon_2}{2}$. Then
$P_x(\tau^x_{n,\varepsilon}>\varepsilon_1,\sigma^x>t_{n_0})\ge\varepsilon_2/{2}$.
Hence
\begin{align*}
E_x\int_0^{\tau^x_{n,\varepsilon}}g(X_r)\,dr&\ge E_x\mathbf{1}_{\{\tau^x_{n,\varepsilon}>\varepsilon_1\}}\int_0^{\varepsilon_1\wedge \sigma^x}g(X_r)\,dr\ge (g(x)-\delta)E_x \mathbf{1}_{\{\tau^x_{n,\varepsilon}>\varepsilon_1\}}\varepsilon\wedge\sigma^x\\&\ge
\varepsilon_1\wedge t_{n_0}(g(x)-\delta)P_x(\tau^x_{n,\varepsilon}>\varepsilon_1,\sigma^x>t_{n_0})\\
&\ge \frac{\varepsilon_2}{2} (g(x)-\delta)\varepsilon_1\wedge t_{n_0},
\end{align*}
in contradiction with (\ref{eq4.cont1}).
\end{proof}

\begin{theorem}
\label{th2.2.2}
Let $u$ be a renormalized solution to \mbox{\rm(\ref{eqi.1})}. Then
\[
\frac 12\lambda_{-k}(u)\rightarrow \mu^-_c,\qquad \frac12\lambda_k(u)\rightarrow \mu^+_c
\]
as $k\rightarrow\infty$ in the narrow topology.
\end{theorem}
\begin{proof}
By Proposition  \ref{prop3.2},
\[
\frac12\langle \lambda_k(u),\eta\rangle\rightarrow  \langle \mu^+_c,\eta\rangle
\]
for every bounded $\eta\in Exc$.  In particular, $\sup_{k\ge 1}\langle \lambda_k(u),1\rangle <\infty$.
Therefore, there exists a subsequence (still denoted by $(k)$) and a positive $\nu^1\in \mathcal M_b(E)$ such that $\frac12\lambda_k(u)\rightarrow \nu^1$
in the vague topology. Let $h\in C_c(E)$. Let  $\{h^1_n\}, \{h^2_n\}$
be sequences satisfying properties asserted in Lemma \ref{lm2.1}.  Then
\[
\langle \mu^+_c, h^2_n\rangle=\lim_{k\rightarrow \infty}\langle \lambda_k(u),h^2_n\rangle  \le
\lim_{k\rightarrow \infty}\langle \lambda_k(u),h\rangle= \langle \nu^1, h\rangle
\]
and
\[
\langle \nu^1, h\rangle=\lim_{k\rightarrow \infty}\langle \lambda_k(u),h\rangle\le \lim_{k\rightarrow \infty}\langle \lambda_k(u),h^1_n\rangle = \langle \mu^+_c, h^1_n\rangle.
\]
Consequently, passing to the limit with $n\rightarrow \infty$ yields $\nu^1=\mu^+_c$. Thus
$\frac12\lambda_k(u)\rightarrow \mu^+_c$ in the vague topology. Since $1$ is an excessive function, we also have
$\langle \frac12\lambda_k(u),1\rangle \rightarrow \langle \mu^+_c,1\rangle$, so $\frac12\lambda_k(u)\rightarrow \mu^+_c$ in the narrow topology. The proof of the second convergence is similar, so we omit it.
\end{proof}

\begin{corollary}
\label{col3.1}
Let $u$ be a renormalized solution to \mbox{\rm(\ref{eqi.1})}. Then
\[
\nu_k\rightarrow \mu_c,\qquad |\nu_k|\rightarrow |\mu_c|
\]
as $k\rightarrow\infty$ in the narrow topology.
\end{corollary}
\begin{proof}
By Proposition \ref{prop4.eq1} and Remark \ref{rem.aqc}, $u$ is quasi-continuous. Therefore $\mathbf{1}_{\{-k<u\le k\}}\cdot\mu_d\rightarrow 0$ in the total variation norm.
Hence,  by Corollary \ref{eq3.sr1} and Theorem \ref{th2.2.2}, $\nu_k\rightarrow \mu_c$ in the narrow topology.
By \cite[Theorem 8.4.7]{Bogachev},
\[
\liminf_{k\rightarrow \infty} |\nu_k|(E)\ge |\mu_c|(E).
\]
On the other hand, using Corollary \ref{eq3.sr1}, Theorem \ref{th2.2.2} and the fact that  $\mathbf{1}_{\{-k<u\le k\}}\cdot\mu_d\rightarrow 0$ in the total variation norm, we get
\[
\limsup_{k\rightarrow \infty} |\nu_k|(E)\le \limsup_{k\rightarrow \infty} \frac12\lambda_k(u)(E)+\limsup_{k\rightarrow \infty} \frac12\lambda_{-k}(u)(E)= |\mu_c|(E).
\]
Thus $\lim_{k\rightarrow \infty} |\nu_k|(E)= |\mu_c|(E)$. From this and
\cite[Theorem 8.4.7]{Bogachev}  we get the desired result.
\end{proof}

\subsection{Second definition}

\begin{definition}
\label{def2.2.cm}
Let $\mu\in\MM_b(E)$.
We say that $u\in\BB(E)$ is a renormalized solution to (\ref{eqi.1}) if
\begin{enumerate}
\item [(i)] $T_k(u)\in D_e(\EE)$ for every $k\ge 0$,
\item [(ii)] for every $k>0$ there exists   $\{\nu_k\}\subset \MM_{0,b}(E)$ such that
\[
\EE(T_k(u),\eta)=\langle \mu_d,\eta\rangle+\langle \nu_k,\eta\rangle,\quad \eta\in D_e(\EE)\cap \BB_b(E),
\]
\item[(iii)]  $\nu_k\rightarrow \mu_c$ in the narrow topology.
\end{enumerate}
\end{definition}

\begin{theorem}
\label{prop3.3}
Let $\mu\in\MM_{b}(E)$.
\begin{enumerate}
\item[\rm(i)] If $u$ is a renormalized solution to \mbox{\rm(\ref{eqi.1})} in the sense of Definition \ref{def2.2}, then $u$ is a renormalized solution
to \mbox{\rm(\ref{eqi.1})} in the sense of Definition \ref{def2.2.cm}.

\item[\rm(ii)] Assume that either $R_\alpha(C_b(E))\subset C_b(E)$ for some (hence for all) $\alpha>0$ and $u\in L^1(E;m)$ or $R_\alpha(\BB_b(E))\subset C_b(E)$ for some (hence for all) $\alpha>0$. If  $u$ is a renormalized solution to
\mbox{\rm(\ref{eqi.1})} in the sense of Definition \ref{def2.2.cm}, then $u$ is a renormalized solution to \mbox{\rm(\ref{eqi.1})} in the sense of Definition \ref{def2.2}.
\end{enumerate}
\end{theorem}
\begin{proof}
Assertion (i) follows from Proposition \ref{prop4.eq1}, Corollary \ref{eq3.sr1} and  Theorem \ref{th2.2.2}
(since $u$ is quasi-continuous it is finite q.e.). To prove (ii), assume that $u$
is a renormalized solution to (\ref{eqi.1}) in the sense of Definition \ref{def2.2.cm}.
By Definition \ref{def2.2.cm}(ii),
\begin{equation}
\label{eq4.44sdx1}
T_k(u)=R\mu_d+R\nu_k,\quad k\ge 0,\quad \mbox{q.e.}
\end{equation}
Therefore $\{R\nu_k\}$ converges q.e. as $k\rightarrow \infty$. Let $v$ denote its limit.
Let $\eta\in \BB(E)$ be a bounded positive function such that $R\eta$ is bounded.
Observe that
\[
\langle |u|,\eta\rangle=\sup_{k\ge 1}\langle T_k(u),\mbox{sign}(u)\eta\rangle\le \langle |\mu_d|,R\eta\rangle+\sup_{k\ge 1}\langle |\nu_k|,R\eta\rangle.
\]
Since $\nu_k$ is narrowly convergent, $\sup_{k\ge 1}\langle |\nu_k|,\eta\rangle <\infty$. Therefore $\langle |u|,\eta\rangle<\infty$
for every $\eta\in \BB^+(E)$  such that $R\eta$ is bounded.
We also have
\begin{equation}
\label{eq4.44sdx}
|R\nu_k|\le R|\mu_d|+|u|.
\end{equation}
Assume that $R_\alpha(\BB_b(E))\subset C_b(E)$ for every $\alpha>0$, i.e. $(R_{\alpha})_{\alpha>0}$ has the strong Feller property.
Then, by (\ref{eq4.44sdx}) and the Lebesgue dominated convergence theorem,
\[
\langle v,\alpha R_\alpha\eta \rangle = \lim_{k\rightarrow \infty}\langle R\nu_k,\alpha R_\alpha \eta\rangle=\lim_{k\rightarrow \infty}\langle \nu_k,\alpha R_\alpha R\eta\rangle
=\langle \mu_c,\alpha R_\alpha R\eta\rangle =\langle R\mu_c,\alpha R_\alpha \eta\rangle.
\]
The third equation follows from strong Feller property. Letting $\alpha\rightarrow \infty$
gives
\[
\langle v,\eta \rangle =\langle R\mu_c, \eta\rangle.
\]
 Since $\eta\in\BB(E)$ was an arbitrary positive function
such that $R\eta$ is bounded, $v=R\mu_c$. Therefore, by (\ref{eq4.44sdx1}),
\[
u=R\mu_d+R\mu_c=R\mu,\quad \mbox{q.e.},
\]
so by Proposition \ref{prop4.eq1}, $u$ is a renormalized solution to (\ref{eqi.1}) in the sense of Definition \ref{def2.2}.
Assume now that $R_\alpha(C_b(E))\subset C_b(E)$ for every $\alpha>0$. For every $\eta\in C_b(E)$ we have
\[
\langle T_k(u),\eta\rangle =\langle \alpha R_\alpha T_k(u),\eta\rangle+\langle \mu_d,R_\alpha\eta\rangle+\langle \nu_k, R_\alpha\eta\rangle.
\]
Letting $k\rightarrow \infty$ and then  $\alpha\searrow 0$ shows that
$u=R\mu$.
By Proposition \ref{prop4.eq1} again, $u$ is a renormalized solution to (\ref{eqi.1}) in the sense of Definition \ref{def2.2}.
\end{proof}

\begin{corollary}
Let $\mu\in\MM_b(E)$ and the  assumptions of Theorem \ref{prop3.3}(ii) hold. Then there exists a unique renormalized solution $u$ to \mbox{\rm(\ref{eqi.1})}
in the sense of Definition \ref{def2.2.cm}. Moreover,
\[
u(x)=\int_E G(x,y)\,\mu(dy),\quad m\mbox{-a.e. }x\in E.
\]
\end{corollary}

\begin{remark}
Even in the case of local operators Definition \ref{def2.2.cm} of
renormalized solutions to (\ref{eqi.1}) is in some cases more
convenient in applications then the other definitions considered
in \cite{DMOP}. For instance,  Petitta, Ponce and Porretta
\cite{Petitta3} applied formulation of this type to solve
evolution equations with smooth measure data and  absorption on
the right-hand side.
\end{remark}

Let $f:E\times\BR\rightarrow \BR$ be a measurable function.  In
\cite{KR:MM} we have proved  a uniqueness result for solutions, in
the sense of Definition \ref{def2.2}, 
to
semilinear equations (\ref{eqi.123}). Thanks to the equivalence
proved in Theorem \ref{prop3.3} we have the uniqueness result for
solutions to (\ref{eqi.123}) in the sense of Definition
\ref{def2.2.cm}. Let us also note here that the existence of
renormalized solutions to semilinear equations (\ref{eqi.123})
with smooth measure data and $f$ satisfying merely the sign
condition with respect to the second variable was proved in
\cite{KR:NoDEA2}. In the case of general measure data the
existence problem for (\ref{eqi.123}) is a very subtle matter. Its
investigation requires introducing the notion of reduced measures
(see \cite{K:CVPDE}).

\begin{definition}
\label{def2.2.cmsl}
Let $\mu\in\MM_b(E)$.
We say that $u\in\BB(E)$ is a renormalized solution to (\ref{eqi.123}) if
\begin{enumerate}
\item [(i)] $T_k(u)\in D_e(\EE)$ for every $k\ge 0$, and $f(\cdot,u)\in L^1(E;m)$,
\item [(ii)] for every $k>0$ there exists   $\{\nu_k\}\subset \MM_{0,b}(E)$ such that
\[
\EE(T_k(u),\eta)=\langle f(\cdot,u),\eta\rangle+\langle \mu_d,\eta\rangle+\langle \nu_k,\eta\rangle,\quad \eta\in D_e(\EE)\cap \BB_b(E),
\]
\item[(iii)]  $\nu_k\rightarrow \mu_c$ in the narrow topology.
\end{enumerate}
\end{definition}

\begin{theorem}
\label{prop3.3sl}
Let $\mu\in\MM_{b}(E)$ and $f$ be non-increasing with respect to the second variable.
\begin{enumerate}
\item[\rm(i)] If  $R_\alpha(C_b(E))\subset C_b(E)$ for some  $\alpha>0$, then there exists at most one
renormalized solution   $u\in L^1(E;m)$ to \mbox{\rm(\ref{eqi.123})} in the sense of Definition \ref{def2.2.cmsl}.

\item[\rm(ii)] If $R_\alpha(\BB_b(E))\subset C_b(E)$ for some $\alpha>0$, then there exists at most one  renormalized solution to \mbox{\rm(\ref{eqi.123})} in the sense of Definition \ref{def2.2.cmsl}.
\end{enumerate}
\end{theorem}
\begin{proof}
Follows from Theorem \ref{prop3.3}, \cite[Theorem 4.4]{KR:MM} and \cite[Corollary 4.3]{K:CVPDE}.
\end{proof}

\section{Structure of renormalized solutions}
\label{sec5}


From now on by saying  renormalized solution we mean a renormalized solution in the sense of Definition \ref{def2.2}.
Recall that by the definition of renormalized solution to (\ref{eqi.1}) we have
\[
-AT_k(u)=\mu_d+\nu_k,
\]
where $\nu_k$ is a bounded smooth measure. By Corollary  \ref{eq3.sr1},
\[
\nu_k=-\mathbf{1}_{\{u>k, u\le -k\}}\mu_d+\frac12(\lambda_k(u)-\lambda_{-k}(u)),
\]
and by  the definition of the measure $\lambda_k(u)$  (see
(\ref{eq2.rdgf1})),
\[
\lambda_k(u)=l_k(u)+j_k(u).
\]
In this section, we study the structure of the measures $l_k(u)$
and $j_k(u)$. We show that $j_k(u)$ has an explicit formula via
$u$ and the kernel of the nonlocal part  of the operator $A$. As
for the measure $l_k(u)$, we  show the so-called
reconstruction formula which is well known for equations with
measure data and local (nonlinear) Leray-Lions type operators
(see, e.g., \cite{DMOP}).

We adopt the notation introduced in  Introduction. For  the
Beurling-Deny decomposition of $(\EE,D(\EE))$ we defer the reader
to \cite{FOT}. Note that for  every   $u\in D_e(\EE)$ there exists
a unique smooth Radon measure $\mu^c_{\langle u\rangle}$ such that
\begin{equation}
\label{eq5.rd12}
\langle \mu^c_{\langle u\rangle},\eta\rangle
= 2\EE^{(c)}(u\eta,u)-\EE^{(c)}(u^2,\eta),\quad \eta\in C_c(E)\cap D(\EE)
\end{equation}
(see \cite[(3.2.20)]{FOT}) By \cite[Lemma 5.3.3]{FOT},
\begin{equation}
\label{eq5.rd1}
\mu^c_{\langle u\rangle}=\RR([u(X)]^c),\quad u\in D_e(\EE),
\end{equation}
where $[u(X)]^c$ is the continuous part of the square bracket  of
the semimartingale $u(X)$.

By the probabilistic definition of a solution to (\ref{eqi.1}),
$u(X)$ is a semimartingale. Therefore $[u(X)]^c$ and hence
$\RR([u(X)]^c)$ are well defined also for solutions to
(\ref{eqi.1}), although in general, solutions are not in $D_e(\EE)$.
In this case, we set $\mu^c_{\langle u\rangle}=\RR([u(X)]^c)$. By
the definition of a renormalized solution to (\ref{eqi.1}),
$T_k(u)\in D_e(\EE)$. By \cite[Lemma 5.6.4]{FOT},
\begin{equation}
\label{eq5.9.1}
\mathbf{1}_{\{-k<u\le k\}}\mu^c_{\langle u\rangle}
=\mu^c_{\langle T_k(u)\rangle}.
\end{equation}
Moreover, by \cite[Lemma 3.2.3]{FOT}, $ \mu^c_{\langle T_k(u)
\rangle}$  is bounded. On the other hand, by (\ref{eq5.rd12}),
for every $\eta\in D_e(\EE)\cap \BB_b(E)$ we have
\[
\langle \mu^c_{\langle T_k(u)\rangle},\eta\rangle
=2\EE^{(c)}(T_k(u)\eta,T_k(u))-\EE^{(c)}(T_k(u)^2,\eta).
\]
By this and (\ref{eq5.9.1}),
\[
\langle \mu^c_{\langle u \rangle},\eta\rangle
=\lim_{k\rightarrow \infty}2\EE^{(c)}(T_k(u)\eta,T_k(u))
-\EE^{(c)}(T_k(u)^2,\eta),\quad \eta\in  D_e(\EE)\cap \BB_b(E).
\]
In general, $ \mu^c_{\langle u \rangle}$ is not  a Radon measure.
However, by (\ref{eq5.9.1}) and the fact that $ \mu^c_{\langle
T_k(u) \rangle}$ is bounded, $\langle  \mu^c_{\langle u
\rangle},|h(u)\eta|\rangle$ is finite for all   $\eta\in\BB_b(E)$ and
$h\in\BB_b(E)$ such that $h$ has compact support.

\begin{proposition}
\label{prop5.3.5} Let $\mu\in \MM_b(E)$ and $u$ be a renormalized
solution to \mbox{\rm(\ref{eqi.1})}.  Then for every $k>0$,
\begin{align*}
j_k(u)(dx)&=2\int_E
\Big(|u(y)-k|-|u(x)-k|-\mbox{\rm sign}(u(x)-k)(u(y)-u(x))\Big)J(dx,dy)\\
&\quad
+(\mathbf{1}_{\{u(x)>k\}}(|k|+k)+\mathbf{1}_{\{u(x)\le
k\}}(|k|-k))\kappa(dx).
\end{align*}
\end{proposition}
\begin{proof}
Let $k>0$ and  $(N,H)$ be a L\'evy system of $\mathbb X$.
By \cite[(A.3.23)]{FOT},
\begin{align*}
{}^p[J^k(u(X))]_t&=
\int_0^t\int_{{E}\cup\{\partial\}}\Big(|u(y)-k|-|u(X_s)-k|\\
&\qquad\qquad\qquad-\mbox{\rm sign}(u(X_s)-k)(u(y)-u(X_s))\Big)
N(X_s,dy)\,dH_s.
\end{align*}
Put $\beta=\RR(H)$. From the above equation it follows that
\begin{align*}
&j_k(u)(dx)\\
&\quad= \int_{{E}\cup\{\partial\}}\Big(|u(y)-k|-|u(x)-k|-\mbox{\rm
sign} (u(x)-k)(u(y)-u(x))\Big) N(x,dy)\,\beta(dx).
\end{align*}
Hence
\begin{align*}
j_k(u)(dx)&=
\int_{E}\Big(|u(y)-k|-|u(x)-k|-\mbox{sign}(u(x)-k)(u(y)-u(x))\Big)
N(x,dy)\,\beta(dx)\\&
+\Big(|u(\partial)-k|-|u(x)-k|-\mbox{sign}(u(x)-k)(u(\partial)-u(x))\Big)
N(x,\{\partial\})\,\beta(dx).
\end{align*}
This implies the desired equality because by the convention
$u(\partial)=0$, and by \cite[Theorem 5.3.1]{FOT},
$N(x,\{\partial\})\,\beta(dx)=\kappa(dx)$ and
$N(x,dy)\beta(dx)=2J(dx,dy)$.
\end{proof}

\begin{proposition}
\label{prop3.5} Let $\mu\in \MM_b(E)$ and $u$ be a renormalized solution to
\mbox{\rm(\ref{eqi.1})}. Then for every bounded $\eta\in Exc$ and
every $\varphi\in \BB^+(E)$,
\begin{equation}
\label{eq.dtbtc}
\int_\BR\langle l_a(u),\eta\rangle\varphi(a)\,da
=\langle \mu^c_{\langle u\rangle},\varphi(u)\eta\rangle.
\end{equation}
\end{proposition}
\begin{proof}
By (\ref{eq2.1}) and the monotone convergence,
\begin{align*}
\int_\BR\langle l_a(u),\eta\rangle\varphi(a)\,da
&=\int_\BR\Big[\lim_{t\searrow 0}\frac1tE_{\eta\cdot m}\int_0^t\,dL^a_r(u(X)) \Big]\varphi(a)\,da\\
&=\lim_{t\searrow 0}\frac1t \Big[\int_\BR E_{\eta\cdot m}\int_0^t\,dL^a_r(u(X)) \Big]\varphi(a)\,da.
\end{align*}
Applying now the occupation time formula  (see
\cite[Corollary 1, page 216]{Protter}) we get
\begin{align*}
\int_\BR\langle l_a(u),\eta\rangle\varphi(a)\,da
=\lim_{t\searrow 0}\frac1t E_{\eta\cdot m}\int_0^t\varphi(u(X_r))d[u(X)]^c_r,
\end{align*}
which by (\ref{eq2.1}) and (\ref{eq5.rd1}) is equal to $\langle \mu^c_{\langle u\rangle},\varphi(u)\eta\rangle$.
\end{proof}

\begin{corollary}
\label{cor5.65} Assume that $(\EE,D(\EE))$ is local (i.e. $J\equiv 0$ in the Beurling-Deny decomposition of $\EE$).
Let $u$ be a renormalized solution to (\ref{eqi.1}). Let $\{b_n\},
\{c_n\}$ be nondecreasing sequences  such that $b_n<c_n$, $n\ge1$
and $b_n\nearrow \infty$, $c_n\nearrow \infty$ as
$n\rightarrow\infty$. Then for every $\eta\in C_b(E)$,
\[
\lim_{n\rightarrow\infty}\frac{1}{c_n-b_n}\langle \mu^c_{\langle u\rangle},
\mathbf{1}_{\{b_n\le u\le c_n\}}\eta\rangle= 2\langle
\mu_c^+,\eta\rangle
\]
and
\[
\lim_{n\rightarrow\infty}\frac{1}{c_n-b_n}\langle \mu^c_{\langle u\rangle},
\mathbf{1}_{\{-c_n\le u\le -b_n\}}\eta\rangle= 2\langle
\mu_c^-,\eta\rangle.
\]
\end{corollary}
\begin{proof}
Since $J\equiv 0$, $\lambda_a(u)=l_a(u)$. Taking now
$\varphi=\frac{1}{c_n-b_n}\mathbf{1}_{[b_n,c_n]}$ or
$\varphi=\frac{1}{c_n-b_n}\mathbf{1}_{[-c_n,-b_n]}$ in
(\ref{eq.dtbtc}), letting $n\rightarrow \infty$ and using Theorem
\ref{th2.2.2} yields the desired convergences.
\end{proof}

\begin{example}
Let $D$ be a bounded domain in $\BR^d$ and $m$ be the Lebesgue
measure on $\BR^d$. Consider the operator
\[
Au=\sum_{i,j=1}^d(a_{ij}u_{x_i})_{x_j},
\]
where $a_{ij}\in L^1_{loc}(D;m)$ and $a=[a_{ij}]_{i,j=1,\dots ,d}$
is a non-negative definite symmetric matrix. To give a precise
definition of the operator $A$, we assume that the form
\[
\EE^0(u,v):= \int_D a\nabla u\nabla v\,dm,\quad u,v\in C^\infty_c(D)
\]
is closable. This is satisfied for instance if $a_{ij}\in H^1_{loc}(D)$ for $i,j=1,\dots, d$ or $a\ge \lambda I$ for some $\lambda>0$ (see, e.g.,  \cite[page
111]{FOT}). Let $(\EE, D(\EE))$ denote the closure of $(\EE^0,C^\infty_c(D))$.  Then there exists a unique self-adjoint operator $(A,D(A))$
such that $D(A)\subset D(\EE)$, and
\[
\EE(u,v)=(-Au,v),\quad u\in D(A), v\in L^2(D;m).
\]
It is clear that $\EE^{(c)}=\EE$, so $(\EE,D(\EE))$ is local. Moreover,
\[
\langle \mu^c_{\langle u\rangle},\eta\rangle =2\int_D\eta|\sigma \nabla u|^2\,dm,
\]
where $\sigma$ is such that $\sigma\cdot\sigma^{T}=a$. By Corollary \ref{cor5.65}, for any $\{b_n\}$, $\{c_n\}$ satisfying its assumptions we have
\[
\lim_{n\rightarrow\infty}\frac{1}{c_n-b_n}\int_{\{b_n\le u\le c_n\}}\eta|\sigma\nabla u|^2\,dm =\langle \mu_c^+,\eta\rangle
\]
and
\[
\lim_{n\rightarrow\infty}\frac{1}{c_n-b_n}\int _{\{-c_n\le u\le -b_n\}}\eta|\sigma\nabla u|^2\,dm =\langle \mu_c^-,\eta\rangle.
\]
\end{example}

\begin{example}
Let $m$ be the Lebesgue measure on $\BR^d$ and $\alpha\in (0,1\wedge(d/2))$. Consider the fractional Laplace operator $A=\Delta^\alpha$  associated with the Dirichlet form on $L^2(\BR^d;m)$ defined as
\[
\begin{cases}
\EE(u,v)=c(\alpha,d)\int_{\BR^d}\int_{\BR^d}
\frac{(u(x)-u(y))(v(x)-v(y))}{|x-y|^{d+2\alpha}}\,dx\,dy, \quad u,v\in D(\EE),
\smallskip\\
D(\EE)=\{u\in L^2(\BR^d;m):\EE(u,u)<\infty,
\end{cases}
\]
where $c(d,\alpha)>0$ is some suitably chosen constant. In this example
$\EE^{c}=0$, $\kappa=0$ and
\[
J(dx,dy)=\frac{c(\alpha,d)}{|x-y|^{d+2\alpha}}\,dx\,dy,
\]
so by Proposition \ref{prop5.3.5},
\[
j_k(u)(dx)=c(d,\alpha)\int_{\BR^d} \frac{|u(y)-k|-|u(x)-k|-\mbox{sign}(u(x)-k)(u(y)-u(x))}{|x-y|^{d+2\alpha}}\,dy\,dx.
\]
\end{example}

\section{Renormalized solutions for smooth measure data}
\label{sec}

In \cite{AAB} a definition of renormalized solutions to (\ref{eqi.1}) with purely jumping
operator on $\BR^d$ and $\mu\in L^1(\BR^d)$ was introduced. In this section, we  show that this definition can  be extended to general  smooth measure data
and the class of operators
considered in the present paper, so in particular to the class of operators considered in \cite{KPU}.
We also show that if $\mu\in\MM_{0,b}$, then renormalized solutions considered in the previous sections are renormalized solutions in the sense of the new definition formulated below.

Given  $h\in C^1_c(\BR)$, $\eta\in D_e(\EE)\cap \BB_b(E)$ and $u\in D(\EE)$  such that  $T_k(u)\in D_e(\EE)$ for $k>0$  and   (\ref{eq.i1.3}) is satisfied, we set
\begin{align}
\label{eq6.2.1}
\nonumber \EE(u,h(u)\eta)&=\EE^{(c)}(T_M(u),h(u)\eta)\\&\quad
\nonumber+\int_{E\times E}(u(x)-u(y))(h(u)(x)-h(u)(y))\frac{\eta(x)+\eta(y)}{2}\,J(dx,dy)\\&
\nonumber\quad+\int_{E\times E}(u(x)-u(y))(\eta(x)-\eta(y))\frac{h(u)(x)+h(u)(y)}{2}\,J(dx,dy)\\&
\quad+\int_E u(x) h(u)(x)\eta(x)\,\kappa(dx),
\end{align}
where $M>0$ is chosen so that  supp$[h]\subset [-M,M]$.  Thanks to the assumptions on $h,\eta$ and $u$
all the integrals appearing in (\ref{eq6.2.1}) are absolutely convergent. Furthermore, by \cite[Theorem 3.2.2]{FOT},
\begin{equation}
\label{eq6.2.2}
\EE^{(c)}(T_M(u)-T_{M'}(u), h(u)\eta)=0
\end{equation}
for every $M'>0$ such that  supp$[h]\subset [-M',M']$, so $\EE(u,h(u)\eta)$ is well defined. Set
\[
\Phi_k(u)=T_{k+1}(u)-T_k(u),\quad k> 0.
\]

\begin{definition}
\label{def6.2.1}
Let $\mu\in\MM_{0,b}(E)$. We say that $u\in\BB(E)$ is a renormalized solution to (\ref{eqi.1}) if
\begin{enumerate}
\item[(i)] $T_k(u)\in D_e(\EE)$ for every  $k>0$ and (\ref{eq.i1.3}) is satisfied,
\item[(ii)]$\EE(u,h(u)\eta)=\langle \mu, h(u)\eta\rangle$  for every $\eta\in\BB_b(E)\cap D_e(\EE)$,

\item[(iii)] $\EE(u,\Phi_k(u))\rightarrow 0$ as $k\rightarrow \infty$.
\end{enumerate}
\end{definition}

\begin{proposition}
Let $\mu\in\MM_{0,b}(E)$. If $u$ is renormalized  solution to
\mbox{\rm(\ref{eqi.1})} in the sense of Definition \ref{def2.2}  then $u$ is
renormalized solution to \mbox{\rm(\ref{eqi.1})} in the sense of Definition
\ref{def6.2.1}.
\end{proposition}
\begin{proof}
By  Definition \ref{def2.2}(ii), for all $k,l>0$,
\[
\EE(T_k(u), T_l(u))=\langle \mu_d,T_l(u)\rangle
+\langle \lambda_k, T_l(u)\rangle\le l\|\mu_d\|_{TV}+l\|\nu_k\|_{TV}.
\]
Hence
\[
\int_{E\times E}(T_k(u)(x)-T_k(u)(y))(T_l(u)(x)-T_l(u)(y))J(dx,dy)
\le l\|\mu_d\|_{TV}+l\|\nu_k\|_{TV}.
\]
Letting $k\rightarrow\infty$ and applying Fatou's lemma and Corollary \ref{col3.1} we get
\[
\int_{E\times E}(u(x)-u(y))(T_l(u)(x)-T_l(u)(y))J(dx,dy)\le l\|\mu\|_{TV},\quad l>0.
\]
From this and condition (i) of Definition \ref{def2.2} it follows that $u$ satisfies condition (i) of  Definition \ref{def6.2.1}. Condition (iii) of Definition \ref{def6.2.1} follows from \cite[Proposition 5.10]{KR:JFA}. As for condition (ii), observe that by condition (ii) of Definition \ref{def2.2} and (\ref{eq6.2.2}), for every $k\ge M$ we have
\begin{align*}
&\EE^{(c)}(T_M(u),h(u)\eta)\\&\quad\quad
+\int_{E\times E}(T_k(u)(x)-T_k(u)(y))(h(u)(x)-h(u)(y))\frac{\eta(x)+\eta(y)}{2}\,J(dx,dy)\\
&\qquad+\int_{E\times E}(T_k(u)(x)-T_k(u)(y))(\eta(x)-\eta(y))\frac{h(u)(x)+h(u)(y)}{2}\,J(dx,dy)\\
&\qquad+\int_E T_M(u)(x) h(u)(x)\eta(x)\,\kappa(dx)=\langle \mu_d,h(u)\eta\rangle+\langle\nu_k,h(u)\eta\rangle.
\end{align*}
Since $|T_k(u)(x)-T_k(u)(y)|\le |u(x)-u(y)|$, applying the Lebesgue dominated convergence theorem shows that the left-hand side of the above equality
tends to $\EE(u,h(u)\eta)$ as $k\rightarrow\infty$. On the other hand, by  Remark \ref{rem4.tvc}, $\lim_{k\rightarrow\infty}\|\nu_k\|_{TV}=0$, which shows that condition (ii) of Definition \ref{def6.2.1} is satisfied.
\end{proof}

\subsection*{Acknowledgements}
{\small This work was supported by Polish National Science Centre
(Grant No. 2017/25/B/ST1/00878).}

\end{document}